\documentclass{amsart}

\newtheorem{theorem}{Theorem}[section]
\newtheorem{lemma}[theorem]{Lemma}
\newtheorem{corollary}[theorem]{Corollary}

\theoremstyle{definition}
\newtheorem{definition}[theorem]{Definition}

\theoremstyle{remark}
\newtheorem{remark}[theorem]{Remark}

\numberwithin{equation}{section}

\begin{document}

\title{Limits of tangents of a quasi-ordinary hypersurface}


\author{Ant\' onio Ara\' ujo}
\address{}
\curraddr{}
\email{ant.arj@gmail.com}
\thanks{}

\author{Orlando Neto}
\address{}
\curraddr{}
\email{orlando60@gmail.com}
\thanks{}


\date{\today}

\dedicatory{}

\commby{}

\begin{abstract}
We compute explicitly the limits of tangents of a quasi-ordinary singularity in terms of its special monomials. We show that the set of limits of tangents of $Y$ is essentially a topological invariant of $Y$. 
\end{abstract}

\maketitle


\section{Introduction}
\noindent
The study of the limits of tangents of a complex hypersurface singularity was mainly developped by Le Dung Trang and Bernard Teissier
(see \cite{LeTe1988} and its bibliography).
Chunsheng Ban \cite{Ban1994} computed the set of limits of tangents $\Lambda$ of a quasi-ordinary singularity $Y$ when
$Y$ has only one very special monomial (see Definition \ref{very_special}).

\noindent
The main achievement of this paper is the explicit computation of the limits of tangents of an arbitrary quasi-ordinary hypersurface singularity
(see Theorems \ref{case_l_1}, \ref{case_g_1} and \ref{case_eq_1}).
Corollaries \ref{cor1}, \ref{cor2} and \ref{cor3}  show that the set of limits of tangents of $Y$ comes quite close to being a topological invariant of $Y$.
Corollary \ref{cor2} shows that $\Lambda$ is a topological invariant of $Y$ when the tangent cone of $Y$ is a hyperplane.
Corollary \ref{cor4} shows that the triviality of the set of limits of tangents of $Y$ is a topological invariant of $Y$.

\noindent
Let $X$ be a complex analytic manifold.
Let $\pi:T^*X\to X$ be the cotangent bundle of $X$.
Let $\Gamma$ be a germ of a Lagrangean variety of $T^*X$ at a point $\alpha$.
We say that $\Gamma$ is in \em generic position \em  if $\Gamma\cap\pi^{-1}(\pi (\alpha))=\mathbb C\alpha$.
Let $Y$ be a hypersurface singularity of $X$.
Let $\Gamma$ be the conormal $T^*_YX$ of $Y$.
The Lagrangean variety $\Gamma$ is in generic position if and only if $Y$ is the germ of an hypersurface with trivial set of limits of tangents.

\noindent
Let $\mathcal M$ be an holonomic $\mathcal D_X$-module.
The characteristic variety of $\mathcal M$ is a Lagrangean variety of $T^*X$.
The characteristic varieties in generic position have a central role in $\mathcal D$-module theory
(cf. Corollary 1.6.4 and Theorem 5.11 of \cite{Ka1981} and Corollary 3.12 of \cite{Neto2001}).
It would be quite interesting to have good characterizations of the hypersurface singularities with trivial set of limits of tangents.
Corollary \ref{cor4} is a first step in this direction.

\noindent
After finishing this paper, two questions arose naturally:

\em
\noindent
Let $Y$ be an hypersurface singularity such that its tangent cone is an hyperplane.
Is the set of limits of tangents of $Y$ a topological invariant of $Y$?

\noindent
Is the triviality of the set of limits of tangents of an hypersurface a topological invariant of the hypersurface?
\em

\medskip

\noindent
Let $p:\mathbb{C}^{n+1}\rightarrow \mathbb{C}^{n}$ be the projection that takes $(x,y)=(x_1,\ldots, x_n, y)$ into $x$. 
Let $Y$ be the germ of a hypersurface of $\mathbb{C}^{n+1}$ defined by $f \in \mathbb{C}\{x_1,\ldots ,x_n,y\}$. 
Let $W$ be the singular locus of $Y$. 
The set $Z$ defined by the equations $f=\partial f / \partial y=0$ is called the \em apparent contour \em of $f$ relatively to the projection $p$. 
The set $\Delta=p(Z)$ is called the \em discriminant \em of $f$ relatively to the projection $p$.

\noindent Near $q \in Y \setminus Z$ there is  one and only one function $\varphi \in \mathcal{O}_{\mathbb{C}^{n+1},q}$ such that $f(x,\varphi(x))=0$. 
The function $f$ defines implicitly $y$ as a function of $x$. Moreover, 
\begin{equation}\label{implicit_derivative}
\frac{\partial y}{\partial x_i} = \frac{\partial \varphi}{\partial x_i} = - \dfrac{\partial f/\partial x_i}{\partial f / \partial y} \text{ on $Y \setminus Z$}.
\end{equation}
Let $\theta=\xi_1 dx_1+\ldots \xi_n dx_n + \eta dy$ be the canonical 1-form of the cotangent bundle 
$T^{*}\mathbb{C}^{n+1}=\mathbb{C}^{n+1}\times\mathbb{C}_{n+1}$. 
An element of the projective cotangent bundle $\mathbb{P}^{*}\mathbb{C}^{n+1}=\mathbb{C}^{n+1}\times \mathbb{P}_n$ i
s represented by the coordinates
\begin{equation*}
(x_1, \ldots, x_n, y;\xi_1: \cdots : \xi_n: \eta).
\end{equation*}
We will consider in the open set $\{ \eta \neq 0 \}$ the chart
\begin{equation*}
(x_1, \ldots, x_n, y,p_1, \ldots, p_n), 
\end{equation*}
where  $p_i=-\xi_i/\eta, 1 \leq i \leq n$. Let $\Gamma_0$ be the graph of the map from $Y \setminus W$ into $\mathbb{P}_n$ defined by 
\begin{equation*}
(x,y) \mapsto \left(\frac{\partial f}{\partial x_1}: \cdots : \frac{\partial f}{\partial x_n} : \frac{\partial f}{\partial y}\right ).
\end{equation*}
Let $\Gamma$ be the smallest closed analytic subset of $\mathbb{P}^*\mathbb{C}^{n+1}$ that contains $\Gamma_0$. The analytic set $\Gamma$ is a Legendrian subvariety of the contact manifold  $\mathbb{P}^*\mathbb{C}^{n+1}$. The projective algebraic set $\Lambda=\Gamma \cap \pi^{-1}(0)$ is called the \em set of limits of tangents \em of $Y$. 
\begin{remark}
It follows from (\ref{implicit_derivative}) that
\begin{equation*}
\left(\frac{\partial f}{\partial x_1}: \cdots : \frac{\partial f}{\partial x_n} : \frac{\partial f}{\partial y}\right )=\left(-\frac{\partial y}{\partial x_1}: \cdots : -\frac{\partial y}{\partial x_n} : 1 \right) \text{ on } Y \setminus Z.
\end{equation*}
\end{remark}
\noindent Let $c_1, \ldots, c_n$ be positive integers. 
We will denote by $\mathbb{C}\{x_1^{1/c_1}, \ldots, x_n^{1/c_n}\}$ the $\mathbb{C}\{x_1, \ldots, x_n\}$ 
algebra given by the immersion from $\mathbb{C}\{x_1, \ldots, x_n\}$ into $\mathbb{C}\{t_1, \ldots, t_n\}$ 
that takes $x_i$ into $t_i^{c_i}, 1 \leq i \leq n$. We set $x_i^{1/c_i}=t_i$. 
Let $a_1, \ldots, a_n$ be positive rationals. Set $a_i=b_i/c_i, 1 \leq i \leq n$, 
where $(b_i,c_i)=1$. 
Given a ramified monomial $M=x_1^{a_1} \cdots x_n^{a_n}=t_1^{b_1} \cdots t_n^{b_n}$ we set 
$\mathcal{O}(M)=\mathbb{C}\{x_1^{1/c_1}, \ldots, x_n^{1/c_n}\}$.

\noindent Let $Y$ be a germ at the origin of a complex hypersurface of $\mathbb{C}^{n+1}$. We say that $Y$ is a quasi-ordinary singularity if $\Delta$ is a divisor with normal crossings. We will assume that there is $l \leq m$ such that $\Delta=\{x_1 \cdots x_l=0\}$.

\noindent If $Y$ is an irreducible quasi-ordinary singularity there are ramified monomials \linebreak $N_0, N_1, \ldots, N_m, g_i \in \mathcal{O}(N_i), 0 \leq i \leq m$, such that $N_0=1$, $N_{i-1}$ divides $N_i$ in the ring $\mathcal{O}(N_i)$, $g_i$ is as unity of $\mathcal{O}(N_i), 1 \leq i \leq m$, $g_0$ vanishes at the origin and the map $x \mapsto (x, \varphi (x))$ is a parametrization of $Y$ near the origin, where
\begin{equation}\label{parametrization}
\varphi=g_0+N_1g_1+\ldots+N_mg_m.
\end{equation}
Replacing $y$ by $y-g_0$, we can assume that $g_0=0$.
The monomials $N_i, 1 \leq i \leq m$, are unique and determine the topology of $Y$ (see \cite{Lipman1988}). They are called the \em special monomials \em of $f$. We set $\tilde{\mathcal{O}}=\mathcal{O}(N_m)$. 
\begin{definition}\label{very_special}
We say that a special monomial $N_i$, $1 \leq i \leq m$, is \em very special \em if $\{N_i=0\}\neq\{N_{i-1}=0\}$.
\end{definition}
\noindent Let $M_1, \ldots, M_g$ be the very special monomials of $f$, where $M_k=N_{n_k}, 1=n_1<n_2<\ldots <n_g, 1 \leq k \leq g$. Set $M_0=1, n_{g+1}=n_g+1$. There are units $f_i$ of $\mathcal{O}(N_{n_{i+1}-1})$, $1 \leq i \leq g$, such that

\begin{equation}\label{very_special_param}
\varphi=M_1f_1+\ldots+M_gf_g.
\end{equation}

\section{Limits of tangents}

\noindent After renaming the variables $x_i$ there are integers $m_k, 1 \leq k \leq g+1$, and positive rational numbers $a_{kij}, 1 \leq k \leq g, 1 \leq i \leq k, 1 \leq j \leq m_k$ such that 

\begin{equation}
M_k=\prod_{i=1}^k \prod_{j=1}^{m_k} x_{ij}^{a_{kij}}, \qquad 1 \leq k \leq g.
\end{equation}
The canonical 1-form of $\mathbb{P}^{*}\mathbb{C}^{n+1}$ becomes 
\begin{equation}
\theta=\sum_{i=1}^{g+1}\sum_{j=1}^{m_i} \xi_{ij}dx_{ij}.
\end{equation}
We set $p_{ij}=-\xi_{ij}/\eta, 1 \leq i \leq g+1, 1 \leq j \leq m_i$. Remark that

\begin{equation}\label{dy_dx}
\frac{\partial y}{ \partial x_{ij}}=a_{iij}\frac{M_i}{x_{ij}}\sigma_{ij},
\end{equation}
where $\sigma_{ij}$ is a unit of $\tilde{\mathcal{O}}$.

\begin{theorem}\label{teo_less_than_one}
If $\sum_{i=1}^{m_1} a_{11i}<1$, $\Lambda \subset \{ \eta=0 \}$.
\end{theorem}
\begin{proof}
Set $m=m_1$, $x_i=x_{1i}$ and $a_i=a_{11i}$, $1 \leq i \leq m$.
Given positive integers $c_1, \ldots, c_m$, it follows from (\ref{dy_dx}) that
\begin{equation}\label{prod_of_p}
\prod_{i=1}^{m} p_{i}^{c_i}=\prod_{i=1}^m x_i^{a_i\sum_{j=1}^{m}c_j-c_i} \phi,
\end{equation}
for some unit $\phi$ of $\tilde{\mathcal{O}}$. By $(\ref{very_special_param})$ and $(\ref{dy_dx})$,
\begin{equation}\label{phi_null_explicit}
\phi (0)=f_1(0)^{\sum_{j=1}^m c_j}\prod_{j=1}^ma_j^{c_j}.
\end{equation}
Hence
\begin{eqnarray}\label{prod2}
\eta^{\sum_{i=1}^m c_i}=\psi \prod_{i=1}^{m} \xi_{i}^{c_i} x_i^{ c_i-a_i\sum_{j=1}^{m}c_j } ,
\end{eqnarray}
for some unit $\psi$.
If there are integers $c_1,\dots,c_m$ such that the inequalities 
\begin{equation}\label{system_a}
\begin{array}{lr}
a_k\sum_{j=1}^m c_j<c_k, & 1 \leq k \leq m,
\end{array}
\end{equation}
hold, the result follows from $(\ref{prod2})$.
Hence it is enough to show that the set $\Omega$ of the m-tuples of rational numbers $(c_1, \ldots, c_m)$ that verify the inequalities $(\ref{system_a})$ is non-empty.
We will recursively define positive rational numbers $l_j,c_j,u_j$ such that
\begin{equation}\label{ineqlcu}
l_j<c_j<u_j, 
\end{equation}
j=1,\ldots,m. Let $c_1, l_1, u_1$ be arbitrary positive rationals verifying $(\ref{ineqlcu})_1$. 
Let $1<s \leq m$. If $l_i,c_i,u_i$ are defined for $i\leq s-1$, set    
\begin{equation}\label{lu}
l_s=\frac{a_s\sum_{j=1}^{s-1}c_j}{1-\sum_{j=s}^m a_j}, ~~~ u_s=(a_s/a_{s-1})c_{s-1}.
\end{equation}
Since $\sum_{j \geq s}a_j<1$ and 
\begin{eqnarray*}
u_s-l_s &=& \frac{a_s}{a_{s-1}(1-\sum_{j=s}^m a_j)}\left ( (1-\sum_{j=s-1}^m a_j)c_{s-1} -a_{s-1}\sum_{j<s-1} c_j\right ) \\ 
&=& \frac{a_s}{a_{s-1}(1-\sum_{j=s}^m a_j)}\left( (1-\sum_{j=s-1}^m a_j)(c_{s-1}-l_{s-1})\right),
\end{eqnarray*}
it follows from $(\ref{ineqlcu})_{s-1}$ that $l_s<u_s$.
\noindent Let $c_s$ be a rational number such that $l_s<c_s<u_s$. Hence $(\ref{ineqlcu})_s$ holds for $s \leq m$.

\noindent Let us show that $(c_1,\ldots,c_m) \in \Omega$. 
Since $c_k<u_k$, then 
$$c_k<\frac{a_k}{a_{k-1}}c_{k-1}, \text{ for } k \geq 2.$$
Then, for $j<k$,
$$c_k<\frac{a_{k}}{a_{k-1}}\frac{a_{k-1}}{a_{k-2}}\cdots\frac{a_{j+1}}{a_j}c_j=\frac{a_{k}}{a_j}c_j.$$
Hence,
\begin{equation}\label{switch}
a_kc_j<a_jc_k, \text{ for }  j>k.
\end{equation}
Since $l_k<c_k$,
\[
a_k\sum_{j=1}^{k-1}c_j<c_k-\sum_{j=k}^m a_jc_k.
\] 
Hence, by $(\ref{switch})$,
\[
a_k\sum_{j=1}^{k-1}c_j<c_k-\sum_{j=k}^m a_kc_j.
\] 
Therefore $a_k\sum_{j=1}^{m}c_j<c_k$.
\end{proof}
\begin{theorem}\label{teoprodnull}
Let $1 \leq k \leq g$. Let $I \subset \{1, \ldots, m_k\}$. Assume that one of the following three hypothesis is verified: 
\begin{enumerate}
\item{$\sum_{j \in I}a_{kkj}>1$\em; \em}
\item{$k=1$, $\sum_{j \in I}a_{11j}=1$ \em and \em $\sum_{j=1}^{m_1}a_{11j}>1$\em; \em}
\item{$k \geq 2$ and $\sum_{j \in I}a_{kkj}=1$\em. \em}
\end{enumerate}
Then $\Lambda \subset \{\prod_{j \in I} \xi_{kj}=0 \}$.
\end{theorem}
\begin{proof}
Case 1: We can assume that $I=\{1,\ldots,n\}$, where $1 \leq n \leq m_k$. Set $a_i=a_{kki}$.
Given positive integers $c_1, \ldots, c_n$, it follows from (\ref{dy_dx}) that
\begin{equation}\label{prod_of_p_b}
\prod_{i=1}^{n} \xi_{ki}^{c_i}=
\prod_{i=1}^{n} x_{ki}^{a_{i}\sum_{j=1}^{n}c_j-c_i} \eta^{\sum_{i=1}^{n}c_i} \varepsilon,
\end{equation}
where $\varepsilon \in \widetilde{\mathcal{O}}$.
Hence it is enough to show that there are positive rational numbers $c_1, \ldots, c_n$ such that
\begin{equation}\label{system_b}
a_k(\sum_{j=1}^{n}c_j)-c_k>0, ~~~ ~~~1\leq k \leq n.
\end{equation}
We will recursively define $l_j,c_j,u_j \in \left ]0,+\infty\right ]$ such that $c_j,l_j \in \mathbb{Q}$, 
\begin{equation}\label{ineqlcu_b}
l_j<c_j<u_j, 
\end{equation}
j=1,\ldots,n, and $u_j \in \mathbb{Q}$ if and only if $\sum_{i=j}^n a_i<1$. Choose $c_1, l_1, u_1$ verifying (\ref{ineqlcu_b}). 
Let $1<s \leq n-1$. Suppose that $l_i,c_i,u_i$ are defined for $1 \leq i\leq s-1$. If $\sum_{j=s}^n a_j<1$, set    
\begin{equation}\label{lu_b}
l_s=(a_s/a_{s-1})c_{s-1}, ~~~ u_s=\frac{a_s\sum_{j=1}^{s-1} c_j}{1-\sum_{j=s}^n a_j}.
\end{equation}

\noindent
Since
\begin{eqnarray*}
u_s-l_s &=& \frac{a_s}{a_{s-1}(1-\sum_{j=s}^n a_j)}\left ( a_{s-1}\sum_{j=1}^{s-2}c_j- c_{s-1}(1-\sum_{j=s-1}^n a_j)\right )\\
& \leq & \frac{a_s}{a_{s-1}(1-\sum_{j=s}^n a_j)}\left ( (1-\sum_{j=s-1}^n a_j) (u_{s-1}-c_{s-1}) \right ),
\end{eqnarray*}
it follows from $(\ref{ineqlcu_b})_{s-1}$ that $l_s<u_s$.

\noindent If $\sum_{j=s}^n a_j\geq 1$, set $l_s$ as above and $u_s=+\infty$.
 
\noindent We choose a rational number $c_s$ such that $l_s<c_s<u_s$. Hence $(\ref{ineqlcu_b})_s$ holds for $1 \leq s\leq n$.

\noindent Let us show that $c_1,\ldots,c_n$ verify $(\ref{system_b})$. We will proceed by induction. First we will show that $c_1,\ldots,c_n$ verify $(\ref{system_b})_n$.
Suppose that $a_n<1$. Since $c_n<u_n$, we have that
$$c_n<\frac{a_n \sum_{j=1}^{n-1}c_j}{1-a_n}.$$
Hence $a_n\sum_{j=1}^nc_j>c_n$.
\noindent If $a_n\geq 1$, then 
$$a_n\sum_{j=1}^n c_j\geq \sum_{j=1}^n c_j>c_n.$$
Hence $(\ref{system_b})_n$ is verified.
\noindent Assume that $c_1,\ldots,c_n$ verify $(\ref{system_b})_k$, $2 \leq k \leq n$. Since $c_k>l_k$,
$$a_k \sum_{j=1}^n c_j>c_k>\frac{a_k}{a_{k-1}}c_{k-1}.$$
Hence $a_{k-1} \sum_{j=1}^n c_j>c_{k-1}$. Therefore $(c_1,\ldots,c_n)$ verify $(\ref{system_b})_{k-1}$.

\noindent Case 2:
Set $a_j=a_{11j}$ and $x_j=x_{1j}$. 
We can assume that $I=\{1,\ldots ,n\}$, where $1 \leq n \leq m_1$.
Given positive integers $c_1, \ldots, c_n$, it follows from (\ref{parametrization}) that

\begin{equation}\label{prodxic}
\prod_{i=1}^n \xi_i^{c_i}= \prod_{i=1}^n x_i^{a_i\sum_{j=1}^nc_j-c_i}\eta^{\sum_{i=1}^nc_i}\varepsilon,
\end{equation}
where $\varepsilon \in \widetilde{\mathcal{O}}$ and $\varepsilon(0)=0$.
Hence it is enough to show that there are positive rational numbers $c_1, \ldots, c_n$, such that
\begin{equation}\label{system_c}
a_k\sum_{j=1}^n c_j=c_k, ~~~ 1\leq k \leq n.
\end{equation}
We choose an arbitrary positive integer $c_1$. Let $1 < s \leq n$. If the $c_i$ are defined for $i<s$, set
\begin{equation}\label{c_s}
c_s=\frac{a_s}{a_{s-1}}c_{s-1}.
\end{equation}
Let us show that $c_1,\ldots,c_n$ verify $(\ref{system_c})$. We will proceed by induction in $k$. First let us show that  $(\ref{system_c})_n$ holds.

\noindent
Let $j<n-1$. By (\ref{c_s}),
\begin{equation}\label{c_n_and_c_j}
c_{n-1}=\frac{a_{n-1}}{a_{n-2}}\frac{a_{n-2}}{a_{n-3}}\cdots \frac{a_{j+1}}{a_{j}}c_j=\frac{a_{n-1}}{a_j}c_j.
\end{equation}
By ($\ref{c_s}$), and since $\sum_{j=1}^n a_j=1$,
\begin{equation*}
c_n=\frac{a_n}{a_{n-1}}c_{n-1}=\frac{c_{n-1}}{a_{n-1}}(1-\sum_{j=1}^{n-1}a_j)=\frac{c_{n-1}}{a_{n-1}}-\sum_{j=1}^{n-1}\frac{a_j}{a_{n-1}}c_{n-1}.
\end{equation*}
Hence, by (\ref{c_n_and_c_j})
\begin{equation*}
c_n=\frac{c_{n-1}}{a_{n-1}}-\sum_{j=1}^{n-1}c_j.
\end{equation*}
Therefore,  $\sum_{j=1}^n c_j=c_{n-1}/a_{n-1}$.
Hence by (\ref{c_s}),
\begin{equation*}
a_n \sum_{j=1}^n c_j=a_n \frac{c_{n-1}}{a_{n-1}}=c_n.
\end{equation*}
Therefore $(\ref{system_c})_n$ holds.

\noindent
Assume $(\ref{system_c})_k$ holds, for $2 \leq k \leq n$. Then
$$a_k\sum_{j=1}^nc_j=c_k=\frac{a_k}{a_{k-1}}c_{k-1}.$$
Hence, $a_{k-1}\sum_{j=1}^nc_j=c_{k-1}$.

\noindent Case 3: 
We can assume that $I=\{1, \ldots, n\}$, where $1 \leq n \leq m_k$. 
Given positive integers $c_1, \ldots, c_n$, it follows from (\ref{dy_dx}) that
\begin{equation*}
\prod_{\i=1}^{n} \xi_{ki}^{c_i}= \left( \prod_{i=1}^{n} x_{ki}^{a_{kki}(\sum_{j=1}^{n}c_j)-c_i} \right) \eta^{\sum_{i=1}^n c_i}\varepsilon,
\end{equation*}
where $\varepsilon \in \tilde{\mathcal{O}}$ and $\varepsilon(0)=0$.
We have reduced the problem to the case 2.
\end{proof}

\begin{theorem}\label{teo_contained_cone}
If $\sum_{k=1}^{m_1}a_{11j}=1$, $\Lambda$ is contained in a cone.
\end{theorem}
\begin{proof}
Set $a_i=a_{11i}, i=1,\ldots m_1$. Given positive integers $c_1, \ldots, c_{m_1}$, there is a unit $\phi$ of $\tilde{\mathcal{O}}$ such that 
\begin{equation}\label{prod_of_p_cone}
\prod_{i=1}^{m_1} \xi_i^{c_i}=(-1)^{\sum_{j=1}^{m_1}c_j}\phi\prod_{i=1}^{m_1} x_i^{\sum_{j=1}^{m_1} c_ja_i-c_i} \eta^{\sum_{j=1}^{m_1} c_j} .
\end{equation}
By the proof of case 2 of Theorem \ref{teoprodnull}, there is one and only one $m_1$-tuple of integers $c_1, \ldots, c_{m_1}$ such that $(c_1,\ldots,c_{m_1})=(1)$, $a_i\sum_{j=1}^{m_1}c_j=c_i, 1 \leq i \leq m_1$, and $\Lambda$ is contained in the cone defined by the equation
\begin{equation}\label{cone}
\prod_{i=1}^{m_1} \xi_i^{c_i}-(-1)^{\sum_{j=1}^{m_1}c_j}\phi(0)\eta^{\sum_{j=1}^{m_1} c_j}=0,
\end{equation}
where $\phi(0)$ is given by $(\ref{phi_null_explicit})$.
\end{proof}
\begin{remark}\label{remvaluation}
Set $D_{\varepsilon}^*=\{x \in \mathbb{C}:0<|x|<\varepsilon\}$, where $0<\varepsilon<<1$. Set $\mu=\sum_{k=1}^{g+1}m_k$. Let $\sigma:\mathbb{C}\rightarrow \mathbb{C}^{\mu}$ be a weighted homogeneous curve parametrized by
\begin{equation*}
\sigma(t)=(\varepsilon_{ki}t^{\alpha_{ki}})_{1 \leq k \leq g+1, 1 \leq i \leq m_k}.
\end{equation*}
Notice that the image of $\sigma$ is contained in $\mathbb{C}^{\mu}\setminus \Delta$. Set $\theta_0(t)=1$ and
\begin{equation*}
\theta_{ki}(t)=\frac{\partial \varphi}{\partial x_{ki}}(\sigma (t), \varphi(\sigma(t))), ~~~1 \leq k \leq g+1, 1 \leq i \leq m_k,
\end{equation*}
for $t \in D_{\varepsilon}^*$. The curve $\sigma$ induces a map from $D_{\varepsilon}^*$ into $\Gamma$ defined by
\begin{equation*}
t \mapsto (\sigma(t), \varphi(\sigma(t));\theta_{11}(t):\cdots:\theta_{g+1,m_g+1}(t):\theta_0(t)).
\end{equation*}
Let $\vartheta:D_{\varepsilon}^*\rightarrow\mathbb{P}^{\mu}$ be the map defined by
\begin{equation}
t \mapsto (\theta_{11}(t):\cdots :\theta_{g+1,m_g+1}(t):\theta_0(t)).
\end{equation}
The limit when $t \to 0$ of $\vartheta(t)$ belongs to $\Lambda$.
The functions $\theta_{ki}$ are ramified Laurent series of finite type on the variable t. Let $h$ a be ramified Laurent series of finite type. If $h=0$, we set $v(h)=\infty$. If $h \neq 0$, we set $v(h)=\alpha$, where $\alpha$ is the only rational number such that $\displaystyle \lim_{t \to 0}t^{-\alpha}h(t)\in \mathbb{C}\setminus\{0\}$. We call $\alpha$ the \em valuation \em of $h$.
Notice that the limit of $\vartheta$ only depends on the functions $\theta_{ki}, \theta_0$ of minimal valuation. Moreover, the limit of $\vartheta$ only depends on the coefficients of the term of minimal valuation of each $\theta_{ij}, \theta_0$. Hence the limit of $\vartheta$ only depends on the coefficients of the very special monomials of $f$. We can assume that $m_{g+1}=0$ and that there are $\lambda_k \in \mathbb{C}\setminus \{0\}, 1\leq k \leq g$, such that 
\begin{equation}
\varphi=\sum_{k=1}^g\lambda_k M_k.
\end{equation}
\end{remark}
\begin{remark}
Let $L$ be a finite set. Set $\mathbb{C}^L=\{(x_a)_{a \in L}: x_a \in \mathbb{C}\}.$ Let $\sum_{a \in L}\xi_a dx_a$ be the canonical 1-form of $T^{*}\mathbb{C}^{L}$. 
Let $\Lambda$ be the subset of $\mathbb{P}_{L}$ defined by the equations
\begin{equation}\label{prod_xi_null}
\prod_{a \in I}\xi_a=0, ~~~I \in \mathcal{I},
\end{equation} 
where $\mathcal{I}\subset \mathcal{P}(L)$. Set $\mathcal{I}^{\prime}=\{J\subset L:J\cap I \neq \emptyset$  for all  $I \in \mathcal{I}\}$, $\mathcal{I}^{*}=\{J\in \mathcal{I}^{\prime}$ such that there is no $K \in \mathcal{I}^{\prime}: K \subset J, K \neq J\}$.
The irreducible components of $\Lambda$ are the linear projective sets $\Lambda_J, J \in \mathcal{I}^{*}$, where $\Lambda_J$ is defined by the equations
\begin{equation*}
\xi_a=0, \qquad a \in J.
\end{equation*}
\end{remark}
\noindent
Let $Y$ be a germ of hypersurface of $(\mathbb{C}^{L},0)$. Let $\Lambda$ be the set of limits of tangents of $Y$. For each irreducible component $\Lambda_J$ of $\Lambda$ there is a cone $V_J$ contained in the tangent cone of $Y$ such that $\Lambda_J$ is the dual of the projectivization of $V_J$. The union of the cones $V_J$ is called the \em halo \em of $Y$. The halo of $Y$ is called "la aur\'eole" of $Y$ in \cite{LeTe1988}. 
\begin{remark}
If $\Lambda$ is defined by the equations (\ref{prod_xi_null}),
the halo of $Y$ equals the union of the linear subsets $V_J, J \in \mathcal{I}^{*}$ of $\mathbb{C}^{\ L}$ , where $V_J$ is defined by the equations
\begin{equation*}
x_a=0, \qquad a \in L \setminus J.
\end{equation*}
\end{remark}
\begin{lemma}\label{lemma_det}
The determinant of the $n \times n$ matrix $(\lambda_i-\delta_{ij})$ equals $$(-1)^n(1-\sum_{i=1}^n \lambda_i).$$ 
\end{lemma}
\begin{proof}
\noindent
Notice that $\det(\lambda_i-\delta_{ij})=$
\begin{equation*}
=       \left|\begin{array}{ccc|c}
                                   & & &1 \\ \
                                    & -I_{n-1} & & \vdots \\
                                    & & & 1\\ \hline
                                   \lambda_1 & \cdots & \lambda_{n-1}& \lambda_n -1\\
            \end{array}\right|=
          \left|\begin{array}{ccc|c}
                                   & & &1 \\ \
                                    & -I_{n-1} & & \vdots \\
                                    & & & 1\\ \hline
                                    0 & \cdots & 0 & \sum_{i=1}^n \lambda_i -1\\
            \end{array}\right|.
\end{equation*}
\end{proof}
\begin{theorem}\label{case_l_1}
Assume that $\sum_{i=1}^{m_1} a_{11i}<1$. Set 
$$L= \cup_{k=2}^{g}\{k\}\times\{1,\ldots,m_k\}, ~~~ \mathcal{I}=\cup_{k=2}^g\{\{k\}\times I: \sum_{j \in I}a_{kkj}\geq 1\}.$$
The set $\Lambda$ is the union of the irreducible linear projective sets $\Lambda_J, J \in \mathcal{I}^{*}$, defined by the equations $\eta=0$ and 
\begin{equation}\label{xinull}
\xi_{kj}=0, ~~~(k,j)\in J.
\end{equation}
\noindent
The tangent cone of $Y$ equals $\{x_{11} \cdots x_{1 {m_1}}=0\}$. The halo of $Y$ is the union of the cones $V_J$, $J \in \mathcal{I}^*$, where $V_J$ is defined by the equations $x_{1j}=0$, $1 \leq j \leq m_1$, and 
\begin{equation}\label{xkjnull}
x_{kj}=0, (k,j) \in L \setminus J.
\end{equation} 
\end{theorem}
\begin{proof}
Let us show that $\Lambda_J\subset \Lambda$.
We can assume that there are integers $n_1,\ldots,n_g$, $1 \leq n_k \leq m_k$, $1 \leq k \leq g$, such that $J=\cup_{k=1}^g \{k\}\times\{n_k+1,\ldots,m_k\}$.
We will use the notations of Remark $\ref{remvaluation}$.

\noindent 
Set $m=\sum_{k=1}^gm_k, n=m-\#J$. Assume that there are positive rational numbers $\alpha_k, \beta_k, 1 \leq k \leq g$, such that 
$\alpha_{ki}=\alpha_k$ if $1 \leq i \leq n_k$, $\alpha_{ki}=\beta_k$ if $n_{k}+1 \leq i \leq m_k$, and $\alpha_k>\beta_k$, $1 \leq k \leq g$. Since $v(\theta_{ki})=v(M_k)-v(x_{ki})=v(M_k)-\alpha_{ki}$, 
\begin{equation*}
\lim_{t \to 0}\vartheta (t) \in \Lambda_J.
\end{equation*}
Let $\psi:(\mathbb{C}\setminus\{0\})^n \rightarrow \Lambda_J$ be the map defined by
\begin{equation}\label{psi}
\psi(\varepsilon_{ij})=\lim_{t \to 0} \vartheta (t).
\end{equation}
The map $\psi$ has components $\psi_{ki}$, $1 \leq i \leq n_k, 1 \leq k \leq g$. In order to prove the Theorem it is enough to show that we can choose the rational numbers $\alpha_{k}, \beta_k$ in such a way that the Jacobian of $\psi$ does not vanish identically. 
\noindent
We will proceed by induction in $k$. Let $k=1$. Since $\sum_{i=1}^{m_1}a_{11i}<1$, $n_1=m_1$. Choose positive rationals $\alpha_1,\beta_1$, $\alpha_1>\beta_1$. There is a rational number $v_0<0$ such that $v(\theta_{1i})=v_0$, for all $1 \leq i \leq n_1$.

\noindent
Assume that there are $\alpha_{k}, \beta_k$ such that $v(\theta_{ki})=v_0$ for $1\leq i \leq n_k$ and $v(\theta_{ki})>v_0$ for $n_k+1 \leq i \leq m_k$,  $k=1,\ldots, u$. 
Set 
\begin{equation*}
\underline{\alpha}_{u+1}=\dfrac{\alpha_u+\sum_{k=1}^u\sum_{i=1}^{m_k}(a_{u+1,k,i}-a_{uki})\alpha_{ki}}{1-\sum_{i=1}^{n_{u+1}}a_{u+1,u+1,i}}.
\end{equation*}
Since the special monomials are ordered by valuation and, by construction of $\Lambda_J$, $\sum_{i=1}^{n_k}a_{kki}<1$ for all $1 \leq k \leq g$, $\underline{\alpha}_{u+1}$ is a positive rational number. Choose  a rational number $\beta_{u+1}$ such that $0<\beta_{u+1}<\underline{\alpha}_{u+1}$. Set   
\begin{equation*}
\alpha_{u+1}=\underline{\alpha}_{u+1}+\frac{\sum_{i=n_{u+1}+1}^{m_{u+1}}a_{u+1,u+1,i}\beta_{u+1}}{1-\sum_{i=1}^{n_{u+1}}a_{u+1,u+1,i}}.
\end{equation*}
Then, $v(\theta_{u+1,i})=v(M_{u+1})-\alpha_{u+1}=v(M_{u})-\alpha_{u}=v_0$ for $1 \leq i \leq n_{u+1}$.

\noindent
Set $\widehat M_{k}=\prod_{i=1}^k \prod_{j=1}^{m_k} \varepsilon_{ij}^{a_{kij}}, 1 \leq i \leq n_k, 1 \leq k \leq g$.
With these choices of $\alpha_{ki}$, we have that
\begin{equation*}
\psi_{ki}=\frac{\widehat M_{k}a_{kki}}{\varepsilon_{ki}}, ~~~  1 \leq i \leq n_k, 1 \leq k \leq g.
\end{equation*}
\noindent
Let $D$ be the jacobian matrix of $\psi$. Since $\partial \psi_{ki}/\partial \varepsilon_{uj}=0$ for all $u>k$, $D$ is upper triangular by blocks. Let $D_k$ be the k-th diagonal block of $D$, $1 \leq k \leq g$. We have that 

\begin{equation*}
D_k=\left( \dfrac{\widehat M_k}{\varepsilon_{ki}\varepsilon_{kj}}a_{kki}(a_{kkj}-\delta_{ij}) \right).
\end{equation*}
By Lemma \ref{lemma_det}, $\det(D_k)=\lambda(1-\sum_{i=1}^{m_k}a_{kki})$ for some $\lambda \in \mathbb{C}\setminus\{0\}$. Hence $\Lambda$ contains an open set of $\Lambda_J$. Since $\Lambda$ is a projective variety and $\Lambda_J$ is irreducible, $\Lambda$ contains $\Lambda_J$.
\end{proof}

\begin{theorem}\label{case_g_1}
Assume that $\sum_{i=1}^{m_1} a_{11i}>1$. Set 
$$L=\cup_{k=1}^{g}\{k\}\times\{1,\ldots,m_k\}, ~~~ \mathcal{I}=\cup_{k=1}^g\{\{k\}\times I: \sum_{j \in I}a_{kkj}\geq 1\}.$$
The set $\Lambda$ is the union of the irreducible linear projective sets $\Lambda_J, J \in \mathcal{I}^{*}$, defined by the equations $(\ref{xinull})$.

\noindent
The tangent cone of $Y$ equals $\{y=0\}$. The halo of $Y$ is the union of the cones $V_J$, $J \in \mathcal{I}^*$, where $V_J$ is defined by the equations $y=0$ and $(\ref{xkjnull})$.
\end{theorem}
\begin{proof}
The proof is analogous to the proof of Theorem \ref{case_l_1}.
On the first induction step we choose \[\beta_1=\left( \dfrac{1-\sum_{i=1}^{n_1}a_{11i}}{\sum_{i=n_1+1}^{m_1}a_{11i}}\right ) \alpha_1 .\]
Hence $\beta_1<\alpha_1$, $ v(\theta_{1i})=v(\eta)=0$ for $1 \leq i \leq n_1$ and $v(\theta_{1i})>0$ for $n_1+1 \leq i \leq m_1$. The rest of the proof proceeds as in the previous case.
\noindent
\end{proof}
\begin{theorem}\label{case_eq_1}
Assume that $\sum_{i=1}^{m_1} a_{11i}=1$. Set 
$$L=\cup_{k=2}^{g}\{k\}\times\{1,\ldots,m_k\}, ~~~ \mathcal{I}=\cup_{k=2}^g\{\{k\}\times I: \sum_{j \in I}a_{kkj}\geq 1\}.$$
 The set $\Lambda$ is the union of the irreducible projective algebraic sets $\Lambda_J, J \in \mathcal{I}^{*}$, where $\Lambda_J$ is defined by the equations $(\ref{cone})$ and $(\ref{xinull})$. 

\noindent
 There are integers $c, d_i$ such that $a_{11i}=d_i/c, 1 \leq i \leq m_1$ and $c$ is the $l.c.d.$ of $d_1, \ldots, d_{m_1}$.
The tangent cone of $Y$ equals 
\begin{equation}\label{tgcone_case_eq_1}
y^c-f(0)^c \prod_{i=1}^{m_1}x_{1i}^{d_i} =0.
\end{equation}

\noindent
 The halo of $Y$ is the union of the cones $V_J$, $J \in \mathcal{I}^*$, where $V_J$ is defined by the equations $(\ref{xkjnull})$ and $(\ref{tgcone_case_eq_1})$.
\end{theorem}
\begin{proof}
Following the arguments of Theorem \ref{case_l_1},
 it is enough to show that $\Lambda_J \subset \Lambda$ for each $J \in \mathcal{I}^*$. 
Choose $J \in \mathcal{I}^*$. 
Let $\tilde \Lambda_J$ be the linear projective variety defined by the equations $(\ref{xinull})$.  
We follow an argument analogous to the one used in Theorem \ref{case_l_1}. We have $n_1=m_1$. 
We choose positive rational numbers $\alpha_1, \beta_1$ such that $\beta_1<\alpha_1$. 
Then $v(\theta_{1i})=0$ for all $i=1, \ldots, m_1$. 
The remaining steps of the proof proceed as before. 
Hence
 \begin{equation*}
\lim_{t \to 0}\vartheta (t) \in \tilde \Lambda_J.
\end{equation*}
Let $\psi: (\mathbb{C}\setminus \{0\})^n\rightarrow \tilde \Lambda_J$ be the map defined by (\ref{psi}).
By Theorem \ref{teo_contained_cone} the image of $\psi$ is contained in $\Lambda_J$.
\noindent
By Lemma \ref{lemma_det}, $\det(D_1)=0$. Let ${D^{\prime}}_1 $ be the matrix obtained from $D_1$ by eliminating the $m_1$-th line and column. 
Then $det(D_1^{ \prime})=\lambda^{ \prime} (1-\sum_{i=1}^{m_1-1}a_{kki})$ for some $\lambda^{ \prime} \in \mathbb{C}\setminus\{0\}$. 
Hence, $\Lambda_J\subset \Lambda$. 
\end{proof}

\noindent
Let $Y$ be a quasi-ordinary hypersurface singularity.

\begin{corollary}\label{cor1}
The set of limits of tangents of $Y$ only depends on the tangent cone of $Y$ and the topology of $Y$.
\end{corollary}

\begin{corollary}\label{cor2}
If the tangent cone of $Y$ is a hyperplane,  the set of limits of tangents of $Y$ only depends on the topology of $Y$.
\end{corollary}

\begin{corollary}\label{cor3}
Let $x_1^{\alpha_1}\cdots x_k^{\alpha_k}$ be the first special monomial of $Y$.
If $\alpha_1+\cdots+\alpha_k\not=1$, the set of limits of tangents of $Y$ only depends on the topology of $Y$.

\end{corollary}

\begin{corollary}\label{cor4}
The triviality of the set of limits of tangents of $Y$ is a topological invariant of $Y$.
\end{corollary}
\begin{proof}
The set of limits of tangents of $Y$ is trivial if and only if all the exponents of all the special monomials of $Y$ are greater or equal than 1.
\end{proof}

\bibliographystyle{amsplain}



\end{document}